\documentclass{article}
\usepackage[utf8]{inputenc}
\usepackage{fullpage}
\usepackage{amsmath}
\usepackage{amsfonts}
\usepackage{amssymb}
\usepackage{tikz-cd}
\usepackage{color,soul}
\usepackage{array}
\usepackage{zref-savepos}
\usepackage{amsthm}
\usepackage{multirow, bigstrut}
\usepackage[font={small,it}]{caption}
\usepackage{tikz}
\usepackage{amsmath}
\usetikzlibrary{arrows}

\usepackage{hyperref}
\hypersetup{
    colorlinks=true,
    linkcolor=blue,
    filecolor=magenta,      
    urlcolor=cyan,
}

\usetikzlibrary{calc}
\usepackage{comment}
\usepackage{enumerate}
\usepackage{tikz}
\usetikzlibrary{shapes.geometric}

\newtheorem{theorem}{Theorem}[section]
\newtheorem*{theorem*}{Theorem}
\newtheorem{lemma}[theorem]{Lemma}
\newtheorem{proposition}[theorem]{Proposition}
\newtheorem{corollary}[theorem]{Corollary}
\newtheorem{conjecture}[theorem]{Conjecture}
\newtheorem{counterexample}[theorem]{Counterexample}

\theoremstyle{definition}
\newtheorem{definition}[theorem]{Definition}

\theoremstyle{plain}

\newcommand{\R}{\mathbb{R}}
\newcommand{\N}{\mathbb{N}}

\DeclareMathOperator{\Symb}{Symb}

\DeclareMathOperator{\maxroot}{maxroot}

\DeclareMathOperator{\diag}{diag}
\DeclareMathOperator{\Ab}{Ab}

\title{On the Further Structure of the Finite Free Convolutions}
\author{Jonathan Leake \\ Nick Ryder \\}
\begin{document}

\maketitle

\begin{abstract}
Since the celebrated resolution of Kadison-Singer (via the Paving Conjecture) by Marcus, Spielman, and Srivastava, much study has been devoted to further understanding and generalizing the techniques of their proof. Specifically, their barrier method was crucial to achieving the required polynomial root bounds on the finite free convolution. But unfortunately this method required individual analysis for each usage, and the existence of a larger encapsulating framework is an important open question. In this paper, we make steps toward such a framework by generalizing their root bound to all differential operators. We further conjecture a large class of root bounds, the resolution of which would require for more robust techniques. We further give an important counterexample to a very natural multivariate version of their bound, which if true would have implied tight bounds for the Paving Conjecture.
\end{abstract}

\section*{Introduction}

There is a long history of studying differential operators that preserve the set of univariate polynomials with only real roots. A classic result in this direction is: given a real-rooted polynomial $p(t)$, it is easy to see that $p(D)$ (where $D = \frac{\partial}{\partial t}$) preserves the set of real-rooted polynomials. When one bounds the degree of the input polynomial though, the set of all differential operators preserving real-rootedness is actually larger than just those of the above form. The set of such differential operators in this case has connections to the classical Walsh \cite{walsh} additive convolution. Namely, if $\boxplus^n$ denotes the Walsh additive convolution for polynomials of degree at most $n$ (also recently known as the finite free convolution; e.g. see \cite{marcus2016polynomial}) then $p(D)$ preserves the set of real-rooted polynomials of degree at most $n$ if and only if there is some real-rooted $q$ such that $p(D)r(t) = (r \boxplus^n q)(t)$ for all $r$.

Recently, there has been interest in understanding how certain differential operators preserving real-rootedness affect the roots of the input polynomial. Much of this interest derives from the notion of interlacing families, heavily studied by Marcus, Spielman, and Srivastava in their collection of papers (\cite{ramanujan},\cite{marcus2015interlacingii},\cite{marcus2015interlacingiv}) containing their celebrated resolution of Kadison-Singer. Most uses of interlacing families share the same loose goal: to study spectral properties of random combinatorial objects. To do this, one equates random combinatorial operations on the objects to differential operators on associated characteristic polynomials. Then, understanding the spectrum of the random objects is reduced to understanding how the roots of certain polynomials are affected by differential operators preserving real-rootedness.

The most robust way to study the effects of a differential operator on roots comes from framework of Marcus, Spielman, and Srivastava. They associate an $S$-transformation to polynomials, inspired from free probability theory, which gives tight bounds on the movement of the largest root via the additive convolution mentioned above. This framework was used in particular in \cite{ramanujan} to prove the existence of Ramanujan bipartite graphs. The strength of their framework is that it gives tight largest root bounds for a general class of differential operator preserving real-rootedness, replacing many of the ad hoc methods used before to study specific desired operators.

That said, some combinatorial objects require the use of multivariate methods to analyze. Here the associated polynomials are real stable (a multivariate generalization of real-rootedness). For these methods, there is no general framework in place to study the effects of the linear operators on the roots. The authors of this paper consider the following to be one of the large unanswered questions around the recent resurgence of interest in finite free convolutions: How does the multivariate additive convolution affect root information?

In an attempt to better understand the multivariate case, we expand upon the previous results of Marcus-Spielman-Srivastava and provide more general results about how \emph{all} roots are of a given polynomial are affected by finite free convolutions. To do this, we first expand their bound on the movement of the largest root to all differential operators preserving real-rootedness. Further, we utilize the theory of hyperbolic polynomials to give more interesting root bounds on interior roots (other roots besides the largest).

With these result in hand, we state a number of conjectures (and some counterexamples) in the direction of stronger univariate results on interior roots and of analogous multivariate results. Proving similar multivariate results seems to be a hard problem in general. But it is the authors' hope that by better fleshing out the details of the additive convolution in the univariate case, one can better abstract to the multivariate case to handle problems such as Kadison-Singer, the Paving conjecture, and Heilman-Lieb root bounds.

\section{The Additive Convolution}
The main focus of this paper is the additive convolution, also referred to as the Walsh convolution and the finite free additive convolution. The convolution can be defined in a coordinate-free way as follows for $p,q \in \R^n[t]$, the space of all univariate polynomials of degree at most $n$. (Again, $D := \frac{\partial}{\partial t}$.)
\[
    p \boxplus^n q = \frac{1}{n!} \sum_{k=0}^n D^k p(t) D^{n-k} q (0)
\]
Notice we get a differential operator if we fix a polynomial $q$ and view the additive convolution as a linear operator $p \mapsto p \boxplus^n q$, and we can obtain all differential operators in this fashion. Some well known properties of the additive convolution are given as follows, where we let $\lambda(p)$ denote the non-increasing vector of roots of $p$ counting multiplicities.

\begin{proposition} \label{prop:add_props}
    Let $p,q \in \R^n[t]$ be real-rooted polynomials of degree at most $n$. We have the following:
    \begin{enumerate}
        \item (Symmetry) $p \boxplus^n q = q \boxplus^n p$
        \item (Shift-invariance) $(p(t+a) \boxplus^n q)(t) = (p \boxplus^n q)(t+a) = (p \boxplus^n q(t+a))(t)$ for $a \in \R^n$
        \item (Scale-invariance) $(p(at) \boxplus^n q(at)) = a^n \cdot (p \boxplus^n q)(at)$ for $a \in \R$
        \item (Derivative-invariance) $(D p) \boxplus^n q = D (p \boxplus^n q) = p \boxplus^n (D q)$ for all $k \in [n]$
        \item (Stability-preserving) $p \boxplus^n q$ is real rooted
        \item (Triangle inequality) $\lambda_1(p \boxplus^n q) \leq \lambda_1(p) + \lambda_1(q)$
    \end{enumerate}
\end{proposition}

Finally, the additive convolution can be used to characterize differential operators which preserve real-rootedness.

\begin{proposition}
    A linear operator $T: \R^n[t] \to \R^n[t]$ is a differential operator which preserves real-rootedness if and only if it can be written in the form $T(p) = p \boxplus^n q$ for some real-rooted $q \in \R^n[t]$.
\end{proposition}

We now discuss a number of stronger properties one can achieve for the additive convolution, using hyperbolicity. Then in \S\ref{sect:result_statement}, we state and discuss our main result: a generalization of the main result from \cite{finiteconvolutions} with a simplified and more intuitive proof.

\subsection{Interior Roots} \label{sect:interior_roots}

The triangle inequality above gives the most basic bound on the largest root of the convolution to two polynomials. The first main collection of interior root bounds can be stated in terms of \emph{majorization}. The majorization order is a partial order on vectors in $\R^n$ which can be thought of morally as saying that the coordinates of one vector are more spread out than the coordinates of the other. Formally, majorization is defined as follows. We refer the reader to \cite{marshall1979inequalities} for more discussion on the following equivalent definitions.
    
\begin{definition} \label{def:maj}
    Given $x,y \in \R^n$, we say that $x$ \emph{majorizes} $y$ and write $y \prec x$ if one of the following equivalent conditions holds. We let $x^\downarrow = (x_1^\downarrow, ..., x_n^\downarrow)$ denote the ordering of the entries $x$ in non-increasing order.
    \begin{enumerate}
        \item $\sum_{i=1}^k y_i^\downarrow \leq \sum_{i=1}^k x_i^\downarrow$ for all $k$, with equality for $k=n$.
        \item $y$ is contained in the convex hull of $\{(x_{\sigma(1)},...,x_{\sigma(n)}) ~|~ \sigma \in S_n\} \subset \R^n$.
        \item There exists a doubly stochastic matrix $D$ (each row and column sum is 1) such that $Dx = y$.
        \item There is a sequence of pinches, of the form $x \mapsto (x_1,...,x_j + \alpha, ..., x_k - \alpha,...,x_n)$ such that the $j^\text{th}$ and $k^\text{th}$ coordinates are getting closer together (without crossing), which takes $x$ to $y$.
    \end{enumerate}
    This makes $\prec$ a partial order on $\R^n$ for all $n$.
\end{definition}

Note that condition (1) applied to the vectors of roots of two polynomials can be interpreted as root bounds involving interior roots. What we need then is some way to prove majorization results about the additive convolution. One way to do this is via hyperbolic polynomials, which enables us to convert inequalities regarding matrix eigenvalues into inequalities regarding roots of polynomials.

\begin{definition}
    Given a homogeneous polynomial $p \in \R[x_1,...,x_n]$ and a vector $e \in \R^n$, we say that $p$ is \emph{hyperbolic with respect to $e$} if $p(e) > 0$ and $p(et + x) \in \R[t]$ is real-rooted for all $x \in \R^n$. Whenever $p$ and $e$ are assumed, we let $\lambda(x)$ to denote the vector of roots of the polynomial $p(et + x)$ in nonincreasing order.
\end{definition}

Hyperbolic polynomials have been heavily studied over the past few decades, starting with \cite{gaarding1959inequality}. There are a number of standard results regarding certain convexity properties of such polynomials, but we omit these here (a good reference is \cite{renegar2006hyperbolic}).

The intuition that one should have when considering hyperbolic polynomials and $\lambda(x)$ is that of the determinant of a matrix and its eigenvalues. This is formalized in the fact that $\det(X)$ (where $X$ is a symmetric matrix of variables) is hyperbolic with respect to the identity matrix, and in this case $\lambda(X)$ is the vector of eigenvalues of $X$. This intuition is further justified by the fact that many properties of the eigenvalues of real symmetric matrices seamlessly transfer over to properties about $\lambda(x)$ for any hyperbolic polynomial $p$. In fact, by exploiting the Helton-Vinnikov theorem (which says that all 3-variable hyperbolic polynomials are determinants \cite{helton2007linear},\cite{helton2007linear}) one can obtain \emph{all} of Horn's inequalities (see \cite{knutson2001honeycombs}) for any hyperbolic polynomial. That is, any inequality that holds between $\lambda(X)$, $\lambda(Y)$, and $\lambda(X+Y)$ for any $X,Y$ real symmetric with $p(X) = \det(X)$ (see Definition \ref{def:horns}) will also hold for any hyperbolic $p(x)$ and any vectors $x,y$. We state this formally as follows. 

\begin{theorem}[\cite{bauschke2001hyperbolic}, \cite{gurvits2004combinatorics}] \label{thm:hyp_horns}
    Fix a hyperbolic polynomial $p$ with respect to $e$. For $v,w \in \R^n$, Horn's inequalities hold for $\lambda(v+w)$, $\lambda(v)$, and $\lambda(w)$. In particular, the following majorization relation holds:
    \[
        \lambda(v+w) \prec \lambda(v) + \lambda(w)
    \]
\end{theorem}

% a result of Borcea and Br{\"a}nd{\'e}n which characterizes the linear operators which preserve root majorization. We state here a corollary of the main result from [?], letting $\lambda(p)$ denote the vector of the roots of $p$ in nonincreasing order.

% \begin{theorem}[Borcea-Br{\"a}nd{\'e}n]
%     Let $T: \R^n[x] \to \R[x]$ be a linear operator which preserves real-rootedness, such that for all $p$ the degree of $T(p)$ equals that of $p$. Then $\lambda(p) \prec \lambda(q)$ implies $\lambda(T(p)) \prec \lambda(T(q))$.
% \end{theorem}

% Note that if $q$ is real-rooted and of degree exactly $n$, then the operator $p \mapsto p \boxplus^n q$ satisfies the premises of this theorem. Notice also that if $p$ is real-rooted and of degree exactly $n$, and if the sum of the roots of $p$ is 0, then $\lambda(x^n) \prec \lambda(p)$. We combine these results to achieve the following.

To apply this result, we need to view the additive convolution as a hyperbolic polynomial. We do this in the following.

\begin{proposition} \label{prop:add_hyp}
    Consider the following, where $\boxplus$ only acts on the $x$ variables.
    \[
        p(x,a_1,...,a_n,b_1,...,b_n) := \left(\prod_{k=1}^n (x-a_k)\right) \boxplus^n \left(\prod_{k=1}^n (x-b_k)\right)
    \]
    Then $p$ is hyperbolic with respect to $e = (1,0,...,0)$.
\end{proposition}
\begin{proof}
    For $q(t) := \prod_k (t - a_k)$, $r(t) := \prod_k (t - b_k)$, and $y = (c,a_1,...,a_n,b_1,...,b_n)$ we compute:
    \[
        p(et + y) = p(t+c,a_1,...,a_n,b_1,...,b_n) = (q \boxplus^n r)(t+2c)
    \]
    So $p(et+y)$ is real-rooted since $q$ and $r$ are and $\boxplus$ preserves real-rootedness. Also, $p(e) = 1 > 0$.
\end{proof}

This fact allows us to immediately apply the previous theorem to the additive convolution. In what follows, we let $\lambda(p)$ denote the vector of roots of $p$ in non-increasing order.

\begin{corollary}
    Let $p,q \in \R^n[x]$ be real-rooted and of degree exactly $n$. Then:
    \[
        \lambda(p \boxplus^n q) \prec \lambda(p) + \lambda(q)
    \]
\end{corollary}
\begin{proof}
    Let $v = (0,a_1,...,a_n,0,...,0)$ and $w = (0,0,...,0,b_1,...,b_n)$, where the $a_k$ and $b_k$ are the roots of $p$ and $q$, respectively. Then $\lambda(v) = \lambda(p)$ and $\lambda(w) = \lambda(q)$. Further, $\lambda(v+w) = \lambda(p \boxplus^n q)$. The result then follows from the previous theorem.
\end{proof}

Note that this immediately gives us interior root inequalities of the following form:
\[
    \sum_{i=1}^k \lambda_i(p \boxplus^n q) \leq \sum_{i=1}^k \lambda_i(p) + \sum_{i=1}^k \lambda_i(q)
\]
Using the same proof as in the corollary, we can similarly obtain all of Horn's inequalities for the roots of $p$, $q$, and $p \boxplus^n q$. For instance, we obtain the Weyl inequalities (for all $i,j$) which more directly bound the interior roots:
\[
    \lambda_{i+j-1}(p \boxplus^n q) \leq \lambda_i(p) + \lambda_j(q)
\]
Whenever $i=j=1$, this boils down to the triangle inequality:
\[
    \lambda_1(p \boxplus^n q) \leq \lambda_1(p) + \lambda_1(q)
\]

As it turns out, Theorem \ref{thm:hyp_horns} above also yields an important majorization preservation result regarding the additive convolution. In \cite{borcea2010hyperbolicity}, Borcea and Br{\"a}nd{\'e}n give a complete characterization of linear operators which preserve majorization of roots. Roughly speaking, the result says that a linear operator $T$ (with certain degree restrictions) which preserves real-rootedness has the following property:
\[
    \lambda(p) \prec \lambda(q) \implies \lambda(T(p)) \prec \lambda(T(q))
\]
Their result then immediately applies to the operator $T_q(p) := p \boxplus^n q$ for any fixed real-rooted $q$. This result also has a nice proof via hyperbolicity, and we demonstrate this now. As a note, the following proof immediately generalizes to any degree-preserving linear operator preserving real-rootedness. It is likely that one could generalize it further to the full Borcea-Br{\"a}nd{\'e}n result, using some of the results regarding polynomial degree from \cite{borcea2010hyperbolicity}.

\begin{corollary} \label{cor:maj_preservation}
    Let $p,q,r \in \R^n[x]$ be real-rooted polynomials of degree exactly $n$ such that $\lambda(p) \prec \lambda(q)$. Then:
    \[
        \lambda(p \boxplus^n r) \prec \lambda(q \boxplus^n r)
    \]
\end{corollary}
\begin{proof}
    Let $a_k$, $b_k$, and $c_k$ be the roots of $p$, $q$, and $r$, respectively. By Definition \ref{def:maj} and the fact that $\lambda(p) \prec \lambda(q)$, we have that $(a_k)$ is in the convex hull of the permutations of $(b_k)$. That is,
    \[
        (a_1,...,a_n) = \sum_{\sigma \in S_n} \beta_\sigma \cdot (b_{\sigma(1)}, ..., b_{\sigma(n)})
    \]
    where $\beta_\sigma \geq 0$ and $\sum_\sigma \beta_\sigma = 1$. With this, we use the following notation:
    \[
        v := (0,a_1,...,a_n,c_1,...,c_n) \qquad w_\sigma := (0,b_{\sigma(1)},...,b_{\sigma(n)},c_1,...,c_n)
    \]
    And so we also have that $v = \sum_{\sigma \in S_n} \beta_\sigma w_\sigma$.
    
    Since $\prec$ is a partial order, we can induct on the majorization relation of Theorem \ref{thm:hyp_horns}, using the hyperbolic polynomial from Proposition \ref{prop:add_hyp}:
    \[
        \lambda(v) = \lambda\left(\sum_{\sigma \in S_n} \beta_\sigma w_\sigma\right) \prec \sum_{\sigma \in S_n} \lambda(\beta_\sigma w_\sigma)
    \]
    By the scale-invariance property of $\boxplus$ (see Proposition \ref{prop:add_props}), we have that $\lambda(\beta_\sigma w_\sigma) = \beta_\sigma \cdot \lambda(w_\sigma)$. This implies:
    \[
        \lambda(p \boxplus^n r) = \lambda(v) \prec \sum_{\sigma \in S_n} \beta_\sigma \cdot \lambda(w_\sigma) = \sum_{\sigma \in S_n} \beta_\sigma \cdot \lambda(q \boxplus^n r) = \lambda(q \boxplus^n r)
    \]
\end{proof}

\subsection{Submodularity and the Main Result} \label{sect:result_statement}

In \cite{finiteconvolutions}, the authors consider the effects of a certain class of differential operators on the largest root of a given real-rooted polynomial:
\[
    U_\alpha := 1 - \alpha D
\]
This differential operator is inspired by the Cauchy transform, via the following equivalence:
\[
    U_\alpha p (x) = 0 \iff  p(x) - \alpha p'(x) = 0
    \iff \frac{p'(x)}{p(x)} = \frac{1}{\alpha} =: \omega
\]
Restricting to points larger than the largest root of $p$, we have that $\frac{p'}{p}$ is a bijection between $(\lambda_1(p), \infty)$ and $(0, \infty)$. Let $\mathcal{K}_\omega(p)$ denote the inverse of $\omega$. Note that as $\omega \to 0$ our inverse tends to infinity, while as $\omega \to \infty$ our inverse tends to $\lambda_1(p)$. Furthermore, $\lambda_1 (U_\alpha p) = \mathcal{K}_\omega(p)$. This definition is inspired by similar objects from free probability, as discussed in \cite{finiteconvolutions}. The main result from \cite{finiteconvolutions} regarding these $U_\alpha$ is given as follows:

\begin{theorem}[\cite{finiteconvolutions}]\label{thm:main_result}
    Let $p,q \in \R^n[t]$ be real-rooted polynomials of degree $n$. For any $\alpha > 0$ we have:
    \[
        \lambda_1(U_\alpha(p \boxplus^n q)) + n \alpha \leq \lambda_1(U_\alpha(p)) + \lambda_1(U_\alpha(q))
    \]
\end{theorem}

As discussed above, every differential operator on polynomials in $\R^n[t]$ can be represented as $T(p) = p \boxplus^n q$ for some polynomial $q \in \R^n[t]$. In particular we can represent $U_\alpha$ via
\[
    U_\alpha(p) = p \boxplus u_\alpha
\]
where $u_\alpha(t) := t^n - n\alpha \cdot t^{n-1}$. Notice that here we have $\lambda_1(u_\alpha) = n\alpha$, which means that the above result can be restated as follows:
\[
    \lambda_1(p \boxplus^n q \boxplus^n u_\alpha) + \lambda_1(u_\alpha) \leq \lambda_1(p \boxplus^n u_\alpha) + \lambda_1(q \boxplus^n u_\alpha)
\]
This is a submodularity relation for the additive convolution. Further, by rearranging this result, it can also be seen as a diminishing returns property of the convolution:
\[
    \lambda_1(p \boxplus^n q \boxplus^n u_\alpha) - \lambda_1(q \boxplus^n u_\alpha) \leq \lambda_1(p \boxplus^n u_\alpha) - \lambda_1(u_\alpha)
\]
The operation $p \mapsto p \boxplus^n q$ can be interpreted as spreading out the roots of $p$ (see the discussion at the beginning of \S\ref{sect:interior_roots}). The above expression then says that, as the roots of a polynomial become more spread out, the operation of convolving by $p$ has less of an effect on the largest root.

The natural next question is: can $u_\alpha$ be replaced by a larger class of real-rooted polynomials in the above expression? The answer is encapsulated in our main result, which says that it can be replaced by any real-rooted polynomial.

\begin{theorem}
    Let $p,q,r \in \R^n[t]$ be real-rooted polynomials of degree $n$. We have:
    \[
        \lambda_1(p \boxplus^n q \boxplus^n r) + \lambda_1(r) \leq \lambda_1(p \boxplus^n r) + \lambda_1(q \boxplus^n r)
    \]
\end{theorem}

To prove this, we adapt and simplify the proof of the original MSS result above. We leave this proof to \S\ref{sect:proof}, where we actually prove slightly more general results.

It is important to note that we were unable to prove this result using the hyperbolicity properties of the additive convolution. This should not be surprising, as morally anything provable for hyperbolic polynomials should come from properties of the eigenvalues of a Hermitian matrix (and this submodularity relation does not hold for matrices in general). We discuss this further in the section when we discuss conjectures regarding interior roots which are analogous to the above theorem.

\section{Strengthening MSS and Associated Conjectures}
Our main result gives an inequality relating the largest (or smallest) roots of additive convolutions of three polynomials, as is done in the MSS paper. The root bound achieved by MSS is crucial to their proof of the paving conjecture, but it is not strong enough to obtain optimal bounds for the paving conjecture. That said, it is believed that root bounds for the interior roots will help to obtain optimal paving bounds. More generally, such root bounds would further clarify how differential operators affect the roots of polynomials.

As we saw in \S\ref{sect:result_statement}, their result can be extended to 3 polynomials in the form of our main result (Theorem \ref{thm:main_result}):
\[
    \lambda_1(p \boxplus^n q \boxplus^n r) + \lambda_1(r) \leq \lambda_1(p \boxplus^n r) + \lambda_1(q \boxplus^n r)
\]
A natural next question becomes: what other inequalities on roots can we achieve in the 3 polynomial case. With this, we arrive at our main collection of conjectures. To simplify the notation, we first make the following definitions.

\begin{definition} \label{def:horns}
    Fix $n \in \N$ and let $I,J,K \subset [n]$. We call $(I,J,K)$ a \emph{Horn's triple} if for all Hermitian $n \times n$ matrices $A,B$ we have:
    \[
        \sum_{i \in I} \lambda_i(A+B) \leq \sum_{j \in J} \lambda_j(A) + \sum_{k \in K} \lambda_k(B)
    \]
    That is, if $(I,J,K)$ give rise to one of Horn's inequalties.
\end{definition}

\begin{definition}
    Fix $n \in \N$ and let $I,L,J,K \subset [n]$. We call $(I,L,J,K)$ a \emph{valid 4-tuple} if for all real-rooted $p,q,r$ of degree $n$ we have:
\[
    \sum_{i \in I} \lambda_i(p \boxplus^n q \boxplus^n r) + \sum_{l \in L} \lambda_l(r) \leq \sum_{j \in J} \lambda_j(p \boxplus^n r) + \sum_{k \in K} \lambda_k(q \boxplus^n r)
\]
\end{definition}

We want to determine all of the valid 4-tuples. It is worth noting that the method of hyperbolic polynomials (which worked for inequalities relating the roots of 2 polynomials) does not work for determining valid 4-tuples. In fact we have the following, even for diagonal matrices:
\[
    \lambda_1(A+B+C) + \lambda_1(C) \not\leq \lambda_1(A+C) + \lambda_1(B+C)
\]
For example, let $A = B = \diag(2,0)$ and $C = \diag(0,2)$.

With these notions in hand, we can now succinctly state our main conjectures. The first is a natural generalization of Horn's inequalities for two polynomials.

\begin{conjecture}
    Let $p,q,r \in \R[x]$ be real-rooted and of degree exactly $n$, and let $(I,J,K)$ be a Horn's triple. Then $(I,I,J,K)$ is a valid 4-tuple.
\end{conjecture}

Note that the indices of the left-hand side of the inequality are the same for both polynomials. But perhaps this does not have to be the case here? That is, can we pick $L \neq I$ such that the inequality for $(I,L,J,K)$ is stronger than
the inequality for $(I,I,J,K)$, and yet it is still a valid 4-tuple?

This turns out to be a difficult question in general. However, we state a few conjectures in this direction. The first would be a negative result if true.

\begin{conjecture}
    Let $p,q,r \in \R[x]$ be real-rooted and of degree exactly $n$, and let $(I,L,J,K)$ be a valid 4-tuple. Then $(I,J,K)$ is a Horn's triple.
\end{conjecture}

Of course you can make the set $L$ a ``weaker'' set of indices than $I$ (meaning that the inequality for $(I,L,J,K)$ is logically weaker than the inequality for $(I,I,J,K)$) to get a new valid 4-tuple. Since such inequalities follow from the conjecture given above, we will ignore these 4-tuples. That said, the only question left is just how much ``stronger'' the set $L$ can be. We give yet another conjecture regarding this question, albeit only in the case where $|I|=|L|=|J|=|K|=1$. To ease notation, we say that $(i,j,k)$ and $(i,l,j,k)$ are a Horn's triple and a valid 4-tuple, respectively (replace singleton sets with the single index).

\begin{conjecture}
    Let $p,q,r \in \R[x]$ be real-rooted and of degree exactly $n$, and let $(i, j, k)$ be a Horn's triple. Note that this is equivalent to $i \geq j+k-1$ (see the Weyl inequalities above, which are strongest Horn's triples of this form). Then $(i, \max(j,k), j, k)$ and $(i, n+1-\max(j,k), j, k)$ are valid 4-tuples.
\end{conjecture}

Notice that for small $j,k$ the first 4-tuple given in the above conjecture is stronger, and for large $j,k$ the second 4-tuple given in above conjecture is stronger.

\section{The Multivariate Case}

All of the root bounds and conjectures discussed in this paper thus far have been for univariate polynomials. However, in their resolution of Kadison-Singer, Marcus-Spielman-Srivastava give bounds on how the \emph{points above the roots} of a given multivariate polynomial change under the action of differential operators. This prompts an obvious question: are there multivariate generalizations of the root bounds discussed in this paper?

To attempt to answer this, we will give the natural multivariate generalization of the additive convolution, along with some basic analogous results. But first we define the points above the roots of a polynomial, and state a few of its properties. This notion should be interpreted as a multivariate generalization of the largest root of a polynomial.

Recall that a polynomial $p \in \R[x_1,...,x_n]$ is \emph{real stable} if it does not vanish when all inputs are in the upper half-plane. We also use $\R^\gamma[x_1,...,x_n]$ with non-negative integer vector $\gamma$ to denote the set of polynomials of degree at most $\gamma_i$ in $x_i$ for all $i$.

\begin{definition}
    For real stable $p \in \R^\gamma[x_1,...,x_n]$, we say that $a \in \R^n$ is \emph{above the roots of $p$} if $p(a + y) \neq 0$ for all $y \in \R_{++}^n$ (strictly positive real vectors). We also let $\Ab(p)$ denote the set of all points above the roots of $p$. Note that this differs slightly from the usual definition, in that $a \in \Ab(p)$ does \emph{not} imply $p(a) \neq 0$. For the sake of simplicity, we say $\Ab(p) = \R^n$ for $p \equiv 0$.
\end{definition}

Note that in the univariate case, $\Ab(p)$ is the interval $[\maxroot(p), \infty)$.

\begin{proposition} \label{prop:ab_connected}
    Let $p \in \R[x_1,...,x_n]$ be real stable. Then $\Ab(p)$ is convex and is the closure of a connected component of the non-vanishing set of $p$.
\end{proposition}
\begin{proof}
    Follows from the theory of hyperbolic polynomials. E.g., see \cite{renegar2006hyperbolic}.
\end{proof}

With the notion of $\Ab(p)$ in hand, we now define and discuss the multivariate version of the additive convolution including a multivariate generalization of the triangle inequality.

\begin{definition}
    For $p,q \in \R^\gamma[x_1,...,x_n]$ we define the bilinear function:
    \[
        (p \boxplus^\gamma q)(x) := \sum_{0 \leq \mu \leq \gamma} \partial_x^\mu p(x) \cdot \partial_x^{\gamma-\mu} q(0)
    \]
\end{definition}

\begin{proposition}
    Let $p,q \in \R^\gamma[x_1,...,x_n]$ be real stable polynomials. We have the following:
    \begin{enumerate}
        \item (Symmetry) $p \boxplus^\gamma q = q \boxplus^\gamma p$
        \item (Shift-invariance) $(p(x+a) \boxplus^\gamma q)(x) = (p \boxplus^\gamma q)(x+a) = (p \boxplus^\gamma q(x+a))(x)$ for $a \in \R^n$
        \item (Scale-invariance) $(p(ax) \boxplus^\gamma q(ax))(x) = a^\gamma \cdot (p \boxplus^\gamma q)(ax)$ for $a \in \R^n$
        \item (Derivative-invariance) $(\partial_{x_k}p) \boxplus^\gamma q = \partial_{x_k} (p \boxplus^\gamma q) = p \boxplus^\gamma (\partial_{x_k} q)$ for all $k \in [n]$
        \item (Stability-preserving) $p \boxplus^\gamma q$ is real stable
        \item (Triangle inequaltiy) $\Ab(p \boxplus^\gamma q) \supseteq \Ab(p) + \Ab(q)$, where $+$ is Minkowski sum
    \end{enumerate}
\end{proposition}
\begin{proof}
    $(1)$, $(2)$, $(3)$ and $(4)$ are straightforward. To prove $(5)$, one can consider the Borcea-Br{\"a}nd{\'e}n symbol (see \cite{bb1}) of the operator
    \[
        \boxplus^\gamma: \R^{(\gamma,\gamma)}[x_1,...,x_n,z_1,...,z_n] \to \R^\gamma[x_1,...,x_n]
    \]
    which is defined on products of polynomials (i.e., simple tensors) via
    \[
        \boxplus^\gamma(p(x)q(z)) := (p \boxplus^\gamma q)(x)
    \]
    and linearly extended. Note that if $\boxplus^\gamma$ preserves stability, then $(5)$ follows as a corollary. That said, the symbol of $\boxplus^\gamma$ takes on a very nice form, using property $(2)$:
    \[
        \Symb_{BB}(\boxplus^\gamma) = \boxplus^\gamma((x+y)^\gamma(z+w)^\gamma) = (x+y)^\gamma \boxplus^\gamma (x+w)^\gamma = (x+y+w)^\gamma
    \]
    This polynomial is obviously real stable, and $(5)$ follows.
    
    To prove $(6)$, we first assume $0 \in \Ab(p) \cap \Ab(q)$ by shifting, since $\Ab(p(x+a)) = \Ab(p) + \{-a\}$. Note also that $0 \in \Ab(p)$ if and only if $p$ has coefficients all of the same sign. (One direction is easy, the other follows by induction and the fact that $\Ab(p) \subseteq \Ab(\partial_{x_i}p)$ by a standard argument.) In this case, $p$ and $q$ have coefficients all of the same sign, and therefore so does $p \boxplus^\gamma q$. That is, in this case $0 \in \Ab(p \boxplus^\gamma q)$.
    
    To complete the proof, we utilize this case to show that $a \in \Ab(p)$ and $b \in \Ab(q)$ implies $a+b \in \Ab(p \boxplus^\gamma q)$. Note that by shifting we have that $0 \in \Ab(p(x+a)) \cap \Ab(q(x+b))$, which implies $0 \in \Ab((p \boxplus^\gamma q)(x+a+b))$ by the previous paragraph. This in turn implies $a+b \in \Ab(p \boxplus^\gamma q)$.
\end{proof}

Again, in the univariate case $\Ab(p)$ is literally the interval $[\maxroot(p), \infty)$. The triangle inequality stated above then is equivalent to the classical version: $\maxroot(p \boxplus^n q) \leq \maxroot(p) + \maxroot(q)$. This is what justifies our calling it ``the triangle inequality''.

The upshot of the previous proposition is that many of the nice classical properties of the univariate convolution are shared with the multivariate additive convolution. That said, it becomes natural to ask a similar question for the stronger results discussed in this paper; that is: what more can we say about how the multivariate additive convolution relates to points above the roots?

Our first conjecture in this direction is a combining of our main theorem (\ref{thm:main_result}) and the multivariate triangle inequality.

\begin{conjecture}
    Let $p,q,r \in \R^\gamma[x_1,...,x_n]$ be real stable. Then:
    \[
        \Ab(p \boxplus^\gamma q \boxplus^\gamma r) + \Ab(r) \supseteq \Ab(p \boxplus^\gamma r) + \Ab(q \boxplus^\gamma r)
    \]
\end{conjecture}

In a (as of yet unpublished) paper of Br{\"a}nd{\'e}n and Marcus, a multivariate analogue of the Marcus-Spielman-Srivastava root bound is given. We believe that this result should follow from the previous conjecture, but it is currently unclear whether or not the methods of Br{\"a}nd{\'e}n-Marcus can be adapted to prove the conjecture itself.

\subsection{A Natural (But False) Conjecture}

It can be shown that the previous conjecture is not enough to prove optimal bounds for the paving conjecture. For this we need something a bit more refined, which we give in the following. This conjecture represents the most natural generalization of the univariate root bound, and the fact that it precisely implies optimal paving bounds only increases its importance. In addition it has been considered independently of the authors by Mohan Ravichandran (personal correspondence; also see \cite{ravichandran2016mixed}) in attempt to prove optimal paving bounds, and this even further suggests its centrality.

Unfortunately though, the conjecture is false in general. We will state it in two equivalent forms, and provide a counterexample.

To do this, we first must relate the notion of $\Ab(p)$ to the notion of \emph{potential} in the multiaffine case. Potential was used by Marcus-Spielman-Srivastava to delicately keep track of root bounds, and so this connection comes at no surprise. We will use the standard definition of potential in what follows:
\[
    \Phi_p^i(a) := \frac{\partial_{x_i}p}{p}(a)
\]

\begin{corollary}
    Let $p \in \R^{(1^n)}[x_1,...,x_n]$ be real stable and multiaffine with $p(0) > 0$ and $0 \in \Ab(p)$. Then:
    \[
        \Phi_p^i(0) \leq 1 \iff -e_i \in \Ab(p)
    \]
\end{corollary}
\begin{proof}
    Since $p(x) > 0$ for $x \in \R_+^n$ and $p$ is multiaffine we have:
    \[
        \Phi_p^i(c \cdot e_i) < 1 \iff 0 < p(c \cdot e_i) - \partial_{x_i}p(c \cdot e_i) = p(0) + (c-1)\partial_{x_i}p(0) = p((c-1) \cdot e_i)
    \]
    It is straightforward that $\Phi_p^i(c \cdot e_i)$ is strictly decreasing in $c$ (or else identically zero) for $c \geq 0$, and therefore:
    \[
        \Phi_p^i(0) \leq 1 \iff \Phi_p^i(c \cdot e_i) < 1 \text{ for all } c > 0
    \]
    Combining these gives:
    \[
        \Phi_p^i(0) \leq 1 \iff p((c-1) \cdot e_i) > 0 \text{ for all } c > 0
    \]
    Note now that $p((c-1) \cdot e_i)$ is linear in $c$, and that $(c-1) \cdot e_i = 0 \in \Ab(p)$ for $c=1$. Therefore Proposition \ref{prop:ab_connected} implies $\Phi_p^i(0) \leq 1$ iff $-e_i \in \Ab(p)$.
\end{proof}

We now state the false conjecture, once in terms of potential and once in terms of points above the roots.

\begin{conjecture}[Strong conjecture, first form (see \cite{ravichandran2016mixed})]
    Let $p,q \in \R^{(1^n)}[x_1,...,x_n]$ be real stable multiaffine polynomials, and let $a$ and $b$ be above the roots of $p$ and $q$ respectively. Suppose for some $\varphi_1, \ldots, \varphi_n \in \R_{++}$, we have the following for all $i \in [n]$:
    \[
        \Phi_p^i(a) \leq \varphi_i, \qquad \Phi_q^i(b) \leq \varphi_i
    \]
    Then for all $i \in [n]$ we have:
    \[
        \Phi_{p \boxplus q}^i\left(a+b-\dfrac{1}{\varphi_i}\right) \leq \varphi_i
    \]
\end{conjecture}

\begin{conjecture}[Strong conjecture, second form]
    Let $p,q \in \R^{(1^n)}[x_1,...,x_n]$ be real stable multiaffine polynomials. Suppose for all $i \in [n]$ we have:
    \[
        -e_i \in \Ab(p), \qquad -e_i \in \Ab(q)
    \]
    The for all $i \in [n]$ we have:
    \[
        -1-e_i \in \Ab\left(p \boxplus^{(1^n)} q\right)
    \]
\end{conjecture}
\begin{proof}[Proof of equivalence.]
    By the previous corollary $-e_i \in \Ab(p)$ is equivalent to $\Phi_p^i(0) \leq 1$. The conclusion of the above conjecture is that $\Phi_{p \boxplus q}^i(-1) \leq 1$. Again by the previous corollary, this is equivalent to $-e_i \in \Ab((p \boxplus q)(x-1))$. This in turn is equivalent to $-1-e_i \in \Ab(p \boxplus q)$.
    
    As a final note, we can restrict to this seemingly less general case (i.e., $\Phi_p^i(0) \leq 1$ instead of $\Phi_p^i(a) \leq \varphi_i$) via shifting and scaling, which completes the proof.
\end{proof}

To disprove this, we give a counterexample to the second formulation. The key idea is to use a polynomial which is extremal with respect to the strongly Rayleigh conditions. These conditions are nice convexity-type properties which are equivalent to real stability for multiaffine polynomials. We recall them now.

\begin{proposition}[\cite{strongrayleigh}]
    Fix multiaffine $p \in \R^{(1^n)}[x_1,...,x_n]$. We have that $p$ is real stable iff for all $x \in \R^n$ and all $i,j \in [n]$ we have:
    \[
        \partial_{x_i}p(x) \cdot \partial_{x_j}p(x) - p(x) \cdot \partial_{x_i}\partial_{x_j}p(x) \geq 0
    \]
\end{proposition}

It is also of interest to note that the polynomial in the above inequality does not depend on $x_i$ or $x_j$. One can see this by taking the partial derivative of the above expression with respect to $x_i$ or $x_j$, recalling that $p$ is multiaffine (this expression will be 0). This makes it relatively easy to determine whether or not 3-variable multiaffine polynomials are real stable, as in the following example.

\begin{counterexample}
    The polynomial
    \[
        p = q = \frac{8}{21} x_1 x_2 x_3 + \frac{80}{21} x_1 x_2 + \frac{27}{7} x_1 x_3 + x_2 x_3 + 4x_1 + 4x_2 + 4x_3 + 4
    \]
    provides a counterexample to the above conjectures.
\end{counterexample}
\begin{proof}
    First we prove that $p=q$ is real stable. By the above comment, we obtain simple expressions for the strongly Rayleigh conditions:
    \[
        \partial_{x_1}p(x) \cdot \partial_{x_2}p(x) - p(x) \cdot \partial_{x_1}\partial_{x_2}p(x) = \frac{1}{21} (7 x_3 + 4)^2
    \]
    \[
        \partial_{x_1}p(x) \cdot \partial_{x_3}p(x) - p(x) \cdot \partial_{x_1}\partial_{x_3}p(x) = \frac{4}{7} (2 x_2 + 1)^2
    \]
    \[
        \partial_{x_2}p(x) \cdot \partial_{x_3}p(x) - p(x) \cdot \partial_{x_2}\partial_{x_3}p(x) = \frac{4}{147} (22 x_1 + 21)^2
    \]
    Notice that all of these expressions are nonnegative for all $x$, which means that $p=q$ is real stable. Also, notice that these polynomials are on the boundary of the set of nonnegative polynomials, and so in some sense $p=q$ is on the boundary of the set of real stable polynomials. Note that this polynomial has 0 above its roots (with $p(0) > 0$), and it is easy to see that $-e_i \in \Ab(p)$ for all $i \in [n]$.
    
    We now compute $p \boxplus q = p \boxplus p$ as follows:
    \[
        p \boxplus q = \frac{64}{441} x_1 x_2 x_3 + \frac{1280}{441} x_1 x_2 + \frac{144}{49} x_1 x_3 + \frac{16}{21} x_2 x_3 + \frac{4768}{147} x_1 + \frac{32}{3} x_2 + \frac{226}{21} x_3 + \frac{1520}{21}
    \]
    Since $\boxplus$ preserves real stability, this polynomial is real stable. Further, we have $0 \in \Ab(p \boxplus q)$ and $(p \boxplus q)(0) > 0$, and so $(p \boxplus q)(x) \geq 0$ for all $x \in \Ab(p \boxplus q)$. With this, we show that $(p \boxplus q)(-1-e_1) < 0$ which contradicts the above conjecture:
    \[
        (p \boxplus q)(-1-e_1) = -\frac{1450}{441}
    \]
\end{proof}

\section{Proof of the Main Result}
\label{sect:proof}
%\begin{proof}[Possible outline for max root part of \ref{submodularity}]

% \begin{lemma}[Pinching]
% Let $p$ have at least 2 distinct roots and $r$ some fixed real rooted polynomial. Order roots $\lambda_1 \geq \ldots \geq \lambda_n$. Then there exists a split of $p$ into real rooted $p = \tilde{p} + \hat{p}$ obtained by (here $\lambda_1 \neq \lambda_k$):
% $$\tilde{p} = (x - \mu)^2 \prod_{i \neq 1, k} (x - \lambda_i) ~~~ (x- \rho)\prod_{i \neq 1, k} (x - \lambda_i)$$
% That satisfy the following:
% \begin{itemize}
% \item $\maxroot{\tilde{p} \boxplus_n r},\maxroot{\hat{p} \boxplus_n r} < \maxroot{p \boxplus_n r}$,
% \item $\lambda_1 > \mu > \lambda_k$
% \item $\rho > \lambda_1$
% \end{itemize}

% \end{lemma}
% \begin{proof}[Sketch of pinching]

% Let $t = \maxroot{p \boxplus_n r}$ and choose $\mu$ such that
% $$\frac{2}{t - \mu} = \frac{1}{t - \lambda_1} + \frac{1}{1-\lambda_k}$$
% Then by the same arguments in MSS we have the second and third item.

% To see the first claim, note that by the second conjecture and our choice of $\mu$ we have $ (t- \lambda(\tilde{p}))^{-1} \prec (t- \lambda(p))^{-1}$ which implies $\lambda(r \boxplus_n p(t-x)) \prec \lambda(r \boxplus_n \tilde{p}(t-x))$. Since shift commutes with additive convolution we get $(t-\lambda(r \boxplus_n \tilde{p})^{-1} \prec (t-\lambda(r \boxplus_n p))^{-1}$. Namely $\maxroot{\tilde{p} \boxplus_n r} < \maxroot{p \boxplus_n r}$.
% \end{proof}

We now set out to prove our main result, Theorem \ref{thm:main_result}. First, we need some standard notation relating polynomials with interlacing roots. If $p,q \in \R[x]$ are real-rooted polynomials with positive leading coefficients, we write $q \ll p$ if $\deg(q) \in \{\deg(p), \deg(p)-1\}$ and the following root inequalities hold for $p$ and $q$ counting multiplicities:
\[
    \cdots \leq \lambda_3(p) \leq \lambda_2(q) \leq \lambda_2(p) \leq \lambda_1(q) \leq \lambda_1(p)
\]
We also say that $q$ \emph{interlaces} $p$. It is a standard fact that $q \ll p$ or $p \ll q$ if and only if $ap + bq$ is real-rooted for all $a,b \in \R$ (assuming positive leading coefficients). Also standard is the fact that $q \ll p$ and $q \ll r$ if and only if $ap + br$ is real-rooted for all $a,b \in \R_+$. It then follows from the theory of interlacing polynomials:

\begin{lemma} \label{lem:interlacing}
    Fix real-rooted $p,q,r \in \R^n[x]$ with positive leading coefficients. If $q \ll p$, then:
    \[
        q \boxplus^n r \ll p \boxplus^n r
    \]
\end{lemma}

Now we introduce the notation used in \cite{finiteconvolutions}. Given a monic polynomial $p$ of degree $n$ with at least 2 distinct roots, we write:
\[
    p(x) = \prod_{i=1}^n (x-\lambda_i)
\]
Order the roots $\lambda_1 \geq \cdots \geq \lambda_n$, and let $k$ be minimal such that $\lambda_1 \neq \lambda_k$. Define $\mu_0 := \frac{\lambda_1+\lambda_k}{2}$ and $\mu_1 := \lambda_1$. Further, for $\mu \in [\mu_0,\mu_1]$ we define:
\[
    \tilde{p}_\mu(x) := (x-\mu)^2 \prod_{i \neq 1,k} (x-\lambda_i)
\]
We then define:
\[
    \hat{p}_\mu(x) := p(x) - \tilde{p}_\mu(x) = ((2\mu-(\lambda_1+\lambda_k))x - (\mu^2 - \lambda_1 \lambda_k)) \prod_{i \neq 1,k} (x-\lambda_i)
\]
For $\mu > \mu_0$, we have that $\hat{p}_\mu$ is of degree $n-1$ with positive leading coefficient and the extra root is at least $\lambda_1$. (Note that when $\mu = \mu_0$, we have that $\hat{p}_\mu$ is of degree $n-2$ with negative leading coefficient.) To see this, notice:
\[
    \rho := \frac{\mu^2 - \lambda_1 \lambda_k}{2\mu-(\lambda_1+\lambda_k)} \geq \lambda_1 \Longleftrightarrow \mu^2 - 2\mu\lambda_1 + \lambda_1^2 = (\mu - \lambda_1)^2 \geq 0
\]
This then implies that for $f_\mu(x) := (x-\mu) \prod_{i \neq 1,k} (x-\lambda_i)$, we have $f_\mu \ll \tilde{p}_\mu$, $f_\mu \ll \hat{p}_\mu$, and $f_\mu \ll p$. In Figure \ref{fig:root_pos}, we illustrate one possibility for the largest roots of these polynomials.

%These interlacing relations hold even when $\mu = \frac{\lambda_1+\lambda_k}{2}$, as then $\hat{p}_\mu$ will have one less root and its leading coefficient will be negative ($\mu^2 \geq \lambda_1\lambda_k$).

\begin{figure}[h] 

\begin{tikzpicture}
\draw[latex-latex] (-3.5,0) -- (3.5,0) ; %edit here for the axis
 % edit here for the vertical lines
\draw[shift={(-2,0)},color=black] (0pt,3pt) -- (0pt,-3pt);
\draw[shift={(-2,0)},color=black] (0pt,0pt) -- (0pt,-3pt) node[below] 
{$\lambda_k$};

\node[] at (-2,.4) {$\ast$};
\draw[shift={(2,0)},color=black] (0pt,3pt) -- (0pt,-3pt);
\draw[shift={(2,0)},color=black] (0pt,0pt) -- (0pt,-3pt) node[below] 
{$\lambda_1$};
\node[] at (2,.4) {$k-1$};

\draw[latex-latex] (-3.5,-1.5) -- (3.5,-1.5) ; %edit here for the axis
 % edit here for the vertical lines
\draw[shift={(-2,-1.5)},color=black] (0pt,3pt) -- (0pt,-3pt);
\draw[shift={(-2,-1.5)},color=black] (0pt,0pt) -- (0pt,-3pt) node[below] 
{$\lambda_k$};
\node[] at (-2,-1.1) {$\ast-1$};

\draw[shift={(.5,-1.5)},color=black] (0pt,3pt) -- (0pt,-3pt);
\draw[shift={(.5,-1.5)},color=black] (0pt,0pt) -- (0pt,-3pt) node[below] 
{$\mu$};
\node[] at (.5,-1.1) {$2$};

\draw[shift={(2,-1.5)},color=black] (0pt,3pt) -- (0pt,-3pt);
\draw[shift={(2,-1.5)},color=black] (0pt,0pt) -- (0pt,-3pt) node[below] 
{$\lambda_1$};
\node[] at (2,-1.1) {$k-2$};
% \draw[*-o] (0.92,0) -- (2.08,0);
% \draw[latex-latex] (-3.5,-1.5) -- (3.5,-1.5) ;\draw[very thick] (0.92,0) -- (1.92,0);
\draw[latex-latex] (-3.5,-3) -- (3.5,-3) ;

\draw[shift={(-2,-3)},color=black] (0pt,3pt) -- (0pt,-3pt);
\draw[shift={(-2,-3)},color=black] (0pt,0pt) -- (0pt,-3pt) node[below] 
{$\lambda_k$};
\node[] at (-2,-2.6) {$\ast-1$};

\draw[shift={(3,-3)},color=black] (0pt,3pt) -- (0pt,-3pt);
\draw[shift={(3,-3)},color=black] (0pt,0pt) -- (0pt,-3pt) node[below] 
{$\rho$};
\node[] at (3,-2.6) {$1$};

\draw[shift={(2,-3)},color=black] (0pt,3pt) -- (0pt,-3pt);
\draw[shift={(2,-3)},color=black] (0pt,0pt) -- (0pt,-3pt) node[below] 
{$\lambda_1$};
\node[] at (2,-2.6) {$k-2$};

\draw[latex-latex] (-3.5,-4.5) -- (3.5,-4.5) ; 

\draw[shift={(-2,-4.5)},color=black] (0pt,3pt) -- (0pt,-3pt);
\draw[shift={(-2,-4.5)},color=black] (0pt,0pt) -- (0pt,-3pt) node[below] 
{$\lambda_k$};
\node[] at (-2,-4.1) {$\ast-1$};

\draw[shift={(.5,-4.5)},color=black] (0pt,3pt) -- (0pt,-3pt);
\draw[shift={(.5,-4.5)},color=black] (0pt,0pt) -- (0pt,-3pt) node[below] 
{$\mu$};
\node[] at (.5,-4.1) {$1$};

\draw[shift={(2,-4.5)},color=black] (0pt,3pt) -- (0pt,-3pt);
\draw[shift={(2,-4.5)},color=black] (0pt,0pt) -- (0pt,-3pt) node[below] 
{$\lambda_1$};
\node[] at (2,-4.1) {$k-2$};

\node[] at (5,0) {$p$};
\node[] at (5,-1.5) {$\tilde{p}_\mu$};
\node[] at (5,-3) {$\hat{p}_\mu$};
\node[] at (5,-4.5) {$f_\mu$};

\end{tikzpicture}
\centering
\caption{Above is an illustration of larger roots of the described polynomials. The labels of the roots are below the number line, while their respective multiplicity is above. $\ast$ is the multiplicity of $\lambda_k$.} \label{fig:root_pos}
\end{figure}
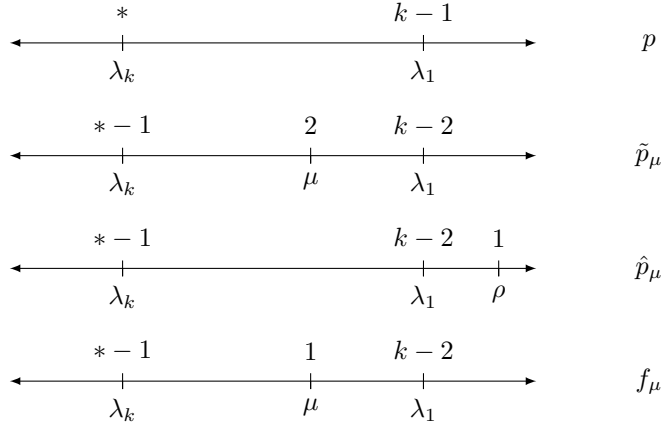

In what follows, we additionally fix a real-rooted $r \in \R[x]$ of degree $n$.

\begin{lemma} \label{lem:ordering}
    Fix any $\mu,\mu'$ with $\mu_0 \leq \mu \leq \mu' \leq \mu_1$ where $\mu_0,\mu_1$ are defined as above. We have:
    \[
        \lambda_1(\tilde{p}_{\mu_0} \boxplus^n r) \leq \lambda_1(p \boxplus^n r) \leq \lambda_1(\tilde{p}_{\mu_1} \boxplus^n r)
    \]
    \[
        \lambda_1(\tilde{p}_\mu \boxplus^n r) \leq \lambda_1(\tilde{p}_{\mu'} \boxplus^n r)
    \]
    \end{lemma}
\begin{proof}
    The first inequality of the first line follows from the fact that the roots of $\tilde{p}_{\mu_0}$ are majorized by that of $p$ (this is because $\tilde{p}_{\mu_0}$ can obtained via a ``pinch'' of the roots of $p$; see property 4 of Definition \ref{def:maj}). The second inequality of the first line follows from Lemma \ref{lem:interlacing} and the fact that $p \ll \tilde{p}_{\mu_1}$. The second line follows from Lemma \ref{lem:interlacing} and the fact that $\tilde{p}_\mu \ll g \ll \tilde{p}_{\mu'}$ for $g(x) := (x-\mu)(x-\mu') \prod_{i \neq 1,k} (x-\lambda_i)$.
\end{proof}

\begin{corollary}
    There exists $\mu \in [\mu_0,\mu_1]$ such that $\lambda_1(\tilde{p}_\mu \boxplus r) = \lambda_1(p \boxplus r)$.
\end{corollary}
\begin{proof}
    The above lemma and continuity.
\end{proof}

Now, let $\mu_*$ denote the maximal $\mu \in [\mu_0,\mu_1]$ such that the previous corollary holds. For simplicity, we will denote $\tilde{p} := \tilde{p}_{\mu_*}$ and $\hat{p} := \hat{p}_{\mu_*}$.

\begin{proposition}
    For $\mu_*$ defined as above, we have that $\mu_* > \mu_0$ and:
    \[
        \lambda_1(\hat{p} \boxplus^n r) = \lambda_1(p \boxplus^n r) = \lambda_1(\tilde{p} \boxplus^n r)
    \]
\end{proposition}
\begin{proof}
    The second equality follows from the definition of $\mu_*$. So we only need to prove the first equality. By linearity $\hat{p} \boxplus^n r$ has a root at $\lambda_1(p \boxplus^n r)$, and so $\lambda_1(\hat{p} \boxplus r) \geq \lambda_1(p \boxplus r)$. So in fact we only need to show that $\lambda_1(\hat{p} \boxplus r) \leq \lambda_1(p \boxplus r)$.
    
    If $\mu_* = \mu_1$, then $\lambda_1(\hat{p}) = \lambda_1(p)$ and $\hat{p} \ll p$. This implies $\lambda_1(\hat{p} \boxplus r) \leq \lambda_1(p \boxplus r)$. Otherwise $\mu_0 \leq \mu_* < \mu_1$. Then for $\mu > \mu_*$, we have $\lambda_1(\tilde{p}_\mu \boxplus r) > \lambda_1(p \boxplus r)$ by Lemma \ref{lem:ordering} which implies $\hat{p}_\mu \boxplus r > 0$ at $\lambda_1(\tilde{p}_\mu \boxplus r)$. Recalling the definition of $f_\mu$ above, $f_\mu \ll \tilde{p}_\mu$ implies $\lambda_1(f_\mu \boxplus r) \leq \lambda_1(\tilde{p}_\mu \boxplus r)$, and $f_\mu \ll \hat{p}_\mu$ implies $\hat{p}_\mu \boxplus r$ has at most one root greater than $\lambda_1(f_\mu \boxplus r)$. Combining all this with the fact that $\hat{p}_\mu \boxplus r$ has positive leading coefficient gives $\lambda_1(\hat{p}_\mu \boxplus r) < \lambda_1(\tilde{p}_\mu \boxplus r)$. Limiting $\mu \to \mu_*$ from above then implies $\lambda_1(\hat{p} \boxplus r) \leq \lambda_1(\tilde{p} \boxplus r) = \lambda_1(p \boxplus r)$.
    
    Now suppose that $\mu_* = \mu_0$, so as to get a contradiction. As $\mu \to \mu_*$ from above, $\hat{p}_\mu \boxplus r$ has positive leading coefficient limiting to zero. So $\hat{p} \boxplus r$ then has one less root, and has negative leading coefficient as discussed above. However, since $\lambda_1(\hat{p}_\mu \boxplus r) < \lambda_1(\tilde{p}_\mu \boxplus r) \leq \lambda_1(\tilde{p}_{\mu_1} \boxplus r)$ for all $\mu > \mu_*$ (as noted earlier in this proof), $\hat{p}_\mu \boxplus r$ must have a root limiting to $-\infty$ as $\mu \to \mu_*$. Therefore the second-from-leading coefficient of $\hat{p}_\mu \boxplus r$ (the sum of negated roots scaled by the leading coefficient) is eventually non-negative as $\mu \to \mu_*$. This contradicts the fact that $\hat{p} \boxplus r$ has negative leading coefficient. (Note that this crucially uses the fact that $\mu_*$ is maximal.)
\end{proof}

The next lemma provides the base case to a more streamlined induction for the proof. In fact, it may even lead to a proof of some sort of majorization relation.

\begin{definition}
    For real-rooted $p \in \R_n[x]$ not necessarily of degree $n$, let $\lambda^n(p) \in \R^n$ be the list of roots of $p$, padded with the mean of the roots, and then ordered in non-increasing order.
\end{definition}

\begin{lemma}
    Fix real-rooted $p,q,r \in \R_n[x]$ such that $\deg(q) = \deg(r) = n$ and $\deg(p) = 1$. Then:
    \[
        \lambda^n(p \boxplus^n q \boxplus^n r) + \lambda^n(r) \prec \lambda^n(p \boxplus^n r) + \lambda^n(q \boxplus^n r)
    \]
\end{lemma}
\begin{proof}
    By shifting, we may assume WLOG that $p,q,r$ all have have roots which sum to 0. Since $\deg(p) = 1$, the result is then equivalent to the following:
    \[
        \lambda^n(D^{n-1}(q \boxplus^n r)) + \lambda^n(r) \prec \lambda^n(D^{n-1}r) + \lambda^n(q \boxplus^n r)
    \]
    Since $\boxplus$ preserves the set of polynomials whose roots sum to 0, this is equivalent to:
    \[
        \lambda^n(r) \prec \lambda^n(q \boxplus^n r)
    \]
    Since $r = x^n \boxplus^n r$ and $\lambda(x^n) \prec \lambda(q)$, the result follows from Corollary \ref{cor:maj_preservation}.
\end{proof}

The following is an immediate corollary of the previous lemma.

\begin{corollary}
    
    Fix real-rooted $p,q,r \in \R_n[x]$ such that $\deg(q) = \deg(r) = n$ and $\deg(p) = 1$. Then:
    \[
        \lambda_1(p \boxplus^n q \boxplus^n r) + \lambda_1(r) \leq \lambda_1(p \boxplus^n r) + \lambda_1(q \boxplus^n r)
    \]
\end{corollary}
% \begin{proof}
%     Since $\deg(p) = 1$, this is equivalent to the following, where $\alpha_p$ is the root of $p$:
%     \[
%         \maxroot(D^{n-1}(q \boxplus^n r)) + \alpha_p + \maxroot(r) \leq \maxroot(D^{n-1}r) + \alpha_p + \maxroot(q \boxplus^n r)
%     \]
%     Letting $\alpha_f$ denote the average of the roots of a given polynomial $f$ (note that this generalizes the definition of $\alpha_p$), this is equivalent to:
%     \[
%         \alpha_{q \boxplus^n r} + \alpha_p + \maxroot(r) \leq \alpha_r + \alpha_p + \maxroot(q \boxplus^n r)
%     \]
%     Now, $\alpha_{q \boxplus^n r} = \alpha_q + \alpha_r$. Substituting this and simplifying gives:
%     \[
%         \alpha_q + \maxroot(r) \leq \maxroot(q \boxplus^n r)
%     \]
%     Since $\alpha_q + \maxroot(r) = \maxroot((x - \alpha_q)^n \boxplus^n r)$ and $q$ majorizes $(x - \alpha_q)^n$, the result follows.
% \end{proof}

We now prove the main result.

\begin{theorem}
    Fix real-rooted $p,q,r \in \R_n[x]$ such that $\deg(q) = \deg(r) = n$ and $\deg(p) = k \leq n$. Then:
    \[
        \lambda_1(p \boxplus^n q \boxplus^n r) + \lambda_1(r) \leq \lambda_1(p \boxplus^n r) + \lambda_1(q \boxplus^n r)
    \]
\end{theorem}
\begin{proof}
    We induct on $k$, using the previous corollary as the base case. Let $p$ be a polynomial of degree $k$ with roots in $[-R,R]$ (for any fixed $R$) which maximizes (by compactness):
    \[
        \beta(p) := \lambda_1(p \boxplus^n q \boxplus^n r) + \lambda_1(r) - \lambda_1(p \boxplus^n r) - \lambda_1(q \boxplus^n r)
    \]
    To get a contradiction, we assume $\beta(p) > 0$. In particular this implies $p$ has at least 2 distinct roots, allowing us to apply the above discussion, notation, and results. By induction we have $\beta(\hat{p}) \leq 0$, which implies:
    \[
        \begin{split}
            \lambda_1(\hat{p} \boxplus^n q \boxplus^n r) &\leq \lambda_1(\hat{p} \boxplus^n r) + \lambda_1(q \boxplus^n r) - \lambda_1(r) \\
                &= \lambda_1(p \boxplus^n r) + \lambda_1(q \boxplus^n r) - \lambda_1(r) \\
                &= \lambda_1(p \boxplus^n q \boxplus^n r) - \beta(p)
        \end{split}
    \]
    Since $\mu_* > \mu_0$ by the previous proposition, $\hat{p}$ has positive leading coefficient. This implies $\tilde{p} \boxplus^n q \boxplus^n r = (p - \hat{p}) \boxplus^n q \boxplus^n r < 0$ when evaluated at $\lambda_1(p \boxplus^n q \boxplus^n r)$. Since $\tilde{p}$ has positive leading coefficient, this gives:
    \[
        \beta(\tilde{p}) - \beta(p) = \lambda_1(\tilde{p} \boxplus^n q \boxplus^n r) - \lambda_1(p \boxplus^n q \boxplus^n r) > 0
    \]
    This contradicts the maximality of $\beta(p)$, since all of the roots of $\tilde{p}$ are contained in $[-R,R]$.
\end{proof}

\begin{corollary}
    Fix real-rooted $p,q,r \in \R_n[x]$. If all polynomials involved are of degree at least 1, then:
    \[
        \lambda_1(p \boxplus^n q \boxplus^n r) + \lambda_1(r) \leq \lambda_1(p \boxplus^n r) + \lambda_1(q \boxplus^n r)
    \]
    Note that the following condition is equivalent to the degree restriction:
    \[
        2n < \deg(p) + \deg(q) + \deg(r) \Longleftrightarrow (n-\deg(p)) + (n-\deg(q)) + (n-\deg(r)) < n
    \]
\end{corollary}
\begin{proof}
    Consider polynomials of degree $n$ whose roots limit to the roots of $p,q,r$ and extra roots limit to $-\infty$. The previous theorem and continuity (and use of Lemma \ref{lem:interlacing} to bound the largest roots away from $+\infty$) then imply the result.
\end{proof}

\section*{Conclusion}

Despite its connections to important problems like the paving conjecture and the entropy conjecture, it is still not fully understood how the additive convolution affects the roots of real-rooted polynomials. In \cite{finiteconvolutions}, Marcus, Spielman, and Srivastava began the study of root movement by investigating the effect of differential operators of the form $I - \alpha D$ on the largest root. In this paper we extended their result to all differential operators which preserve real-rootedness. This extension alone doesn't have any immediate applications we are aware of. 

The resolution of Horn's conjecture by Knutson and Tao (see \cite{knutson2001honeycombs}) gave a full characterization of the eigenvalues of the sum of two Hermitian matrices. We were able to obtain Horn's inequalities for the additive convolution as well via hyperbolicity, but understanding the full effect of the additive convolution on roots remains a mystery. The entropy conjecture, which quantifies the effect of the additive convolution on the discriminant of a polynomial, is one approach to understanding the effect of the roots holistically. Our submodular majorization (and generalized Horn's inequalities) conjectures provide another insight into the workings of the inner roots. Because submodularity is unique to the additive convolution, we believe it will require a new framework (beyond traditional hyperbolicity tools) to tackle these conjectures.

Another possible future direction is extending submodularity results to the $b$-additive convolution, in which derivatives are replaced by certain finite differences. Such convolutions have an intimate connection to the \emph{mesh} of a real-rooted polynomial, which is the minimal distance between any two roots (e.g., see \cite{finitemesh} and \cite{leake2017connecting}). In our testing we found several submodularity relations among such $b$-additive convolutions. The additive convolution can be obtained by limiting $b \to 0$, and so any results for the $b$-additive convolution are strictly stronger than the conjectures in this paper. The advantage of trying to prove these statements in the finite difference case comes in the limited structures available: fewer operations interact nicely with the mesh of a polynomial compared to those operations which preserve real-rootedness, and this may better direct the study of the roots.

Finally in the multivariate realm, little is known. And, many of the natural extensions of these results seem to fail in the multivariate case. The state of the art in this direction is currently the ad hoc barrier function arguments used by MSS in their resolution of Kadison-Singer. That said, an important next step for their work is to encapsulate their techniques in a more coherent theory. We believe that our results and conjectures are a step in the right direction.

\bibliographystyle{amsalpha}
\bibliography{main}

% \[p \boxtimes^n r = \frac{1}{n!}\sum_{\sigma \in \text{S}_n} \prod_i (x - a_i b_{\sigma(i)})\]
% Since $\binom{n}{k}_q q^{\binom{k}{2}} = \sum_{S \subset [n], |S| = k} q^{\Sigma S}$ where $\Sigma S$ is the sum of the elements in $S$.
% \begin{align*}p \boxtimes^n_q r &= \sum_{k=0}^n \binom{n}{k}_q^{-1} (-1)^k q^{-\binom{k}{2}}p_k r_k x^k\\
% &=  \sum_{k=0}^n \left(\sum_{S \subset [n], |S| = k} q^{\Sigma S} \right)^{-1} \left( \sum_{S \subset [n], |S| = n-k} a_k \right) \left( \sum_{S \subset [n], |S| = n-k} b_k \right) (-x)^k\\
% \end{align*}

% \begin{align*}p \boxtimes^n_q r &= \frac{1}{(n)_q!}\sum_{k=0}^n  (k)_q!(n-k)_q!  q^{-\binom{k}{2}}p_k r_k (-x)^k\\ &=\frac{1}{(n)_q!}\sum_{k=0}^n  (k)_{q^{-1}}!(n-k)_q!  p_k r_k (-x)^k\\
% &=\frac{1}{(n)_q!}\sum_{k=0}^n  (k)_{q^{-1}}!(n-k)_q! e_{n-k}(\mathbf{a}) e_{n-k}(\mathbf{b}) (-x)^k\\
% \end{align*}

\end{document}